\documentclass[11pt, oneside]{article}   	
\usepackage{amssymb, amsmath, amsthm, newtxmath, mathtools, amsfonts, graphicx,color}
\usepackage{geometry,hyperref,amsmath}                		
\usepackage{graphicx}				
\usepackage{amssymb}

\newtheorem{theorem}{Theorem}[section]
\newtheorem{lemma}[theorem]{Lemma}
\newtheorem{corollary}[theorem]{Corollary}

\theoremstyle{remark}
\newtheorem{ex}[theorem]{Example}

\renewcommand{\H}{\mathcal{H}}
\newcommand{\R}{\mathbb{R}}
\newcommand \Z {\mathbb Z}
\newcommand \N {\mathbb N}

\newcommand \1 {\vmathbb{1}}

\renewcommand \sl {\mathrm{SL}}
\newcommand \gl {\mathrm{GL}}

\newcommand{\x}{{\vec x}}
\newcommand{\p}{{\vec p}}
\newcommand{\y}{{\vec y}}
\newcommand{\vol}{\mathrm{vol}}
\newcommand{\eqnref}[1]{(\ref{#1})}

\title{Density of Visible Lattice Points on Hyperplanes and their Intersections}
\author{Finnley Goss and Kelly McKinnie}

\begin{document}
\maketitle

\begin{abstract}
    A lattice point $\x=(x_1,\dots,x_n)\in\Z^{n}$ is said to be visible if the line segment between $\x$ and the origin contains no other lattice point. 
    In this paper, we compute the asymptotic density of visible lattice points on hyperplanes and their intersections. In particular, we show that the hyperplane $\vec a \cdot \x = b$ in $\R^{n}$ has visible point density $J_{n-1}(b)/b^{n-1}$ where $J$ is the Jordan totient function.  We extend this basic result to find the density of visible points on the intersection of hyperplanes and to the density of $k$-th power free points. Finally, for a fixed dimension $n$, we consider the closure of the set of all possible densities that occur.

\end{abstract}

\section{Introduction} 
\label{sec1}
A lattice point $\x=(x_1,...,x_n)\in\Z^n$ is said to be visible if the line segment between $\x$ and the origin intersects no other lattice points. Equivalently, a lattice point is visible if and only if $\gcd(\x)=\gcd(x_1,...,x_n)=1$. The density of visible lattice points in a bounded, $n$-dimensional region is given by the ratio of visible to total lattice points contained in the region. The asymptotic density of visible lattice points in $\Z^n$ is defined to be the limit of this bounded density as the diameter approaches infinity. It has long been known that the asymptotic density of visible lattice points in $\Z^n$ is ${1}/{\zeta(n)},$ where $\zeta$ is the Riemann-Zeta function (\cite{Nymann}, see also the classic references \cite{Apostol} and \cite{HardyWright}).

Numerous generalizations of the classic notion of density of visible points exist in the literature (e.g., \cite{Goins_2018}, \cite{lu_2023}, \cite{berend_2024}). In this paper we define and calculate the asymptotic density of visible lattice points which lie on hyperplanes and their intersections. In our main result, Theorem \ref{snfdensity}, we provide a formula for the density on any such domain. The proof uses $\gl_n(\Z)$ transformations of the domain, the Smith Normal Form and the sieve methods found in \cite{Counting-Problems} adapted to our setting. Let $a_i,b \in \Z$ with $\gcd(a_i)=1$. In Corollary \ref{generalhyperplanes} we show that the asymptotic density of visible points on the hyperplane $a_1x_1+\cdots+a_nx_n = b$ is $J_{n-1}(b)/b^{n-1} = \prod_{p|b}1-{1}/{p^{n-1}}$ where the $J$ is the Jordan totient and the product is over all primes dividing $b$.

In Section \ref{k-free sec}, we generalize the arguments used in Theorem \ref{snfdensity} to compute the density of $k$-visible lattice points on the same domains. For $k\in\Z_{>0}$, a lattice point $\x\in\Z^n$ is {\it $k$-visible} if $p^k\nmid\gcd(x)$ for all prime numbers $p$. For $k=1$, these are the visible lattice points. For $k=2$, these are the lattice points with square-free greatest common divisors. 

Lastly, in Section \ref{sec:ex} we consider the set of asymptotic densities of all hyperplanes in $\R^n$. For $n=2$, the densities of lines are all numbers of the form $\varphi(b)/b$ for $b \in \Z$. These numbers are dense in $[1/\zeta(1),0] = [0,1]$ by \cite{Schinzel}. In Theorem \ref{achievement set} we show that for $n>2$ the set of densities of visible points on hyperplanes is not dense in the interval $[1/\zeta(n-1),1]$.

\subsection{Preliminaries}
Fix $n \in \N$ and let $\Delta \subseteq \R^n$ be a nonempty set. To study the asymptotic density of visible lattice points in $\Delta\cap\Z^n$ we define the following subsets of $\Delta$:
\begin{align}
    \Delta(r)&=\Delta\cap[-r,r]^n\\
    \Delta_1(r)&=\Delta(r)\cap\Z^n\\
    \Delta_b(r)&=\{\x\in\Delta_1(r):\gcd(\x,b)=1\}
\end{align}
so that $\Delta_1(r)$ is the subset of lattice points in $\Delta(r)$ and $\Delta_b(r)$ is the subset of those lattice points that are coprime to $b$, that is, $\gcd(\x,b) = \gcd(x_1,\ldots,x_n,b)=1$. In particular, $\Delta_0(r)$ is the set of visible lattice points in $\Delta_1(r)$. We use $|\cdot |$ to indicate the size of a set. 

The asymptotic density of visible lattice points in $\Delta$, if it exists, is defined to be
\[\rho(\Delta)=\lim_{r\to\infty}\frac{|\Delta_0(r)|}{|\Delta_1(r)|}.\]

Taking $\Delta=\R^n$, the classic result is $\rho(\R^n)=1/\zeta(n)$ (\cite{Nymann}). In this paper we concern ourselves with $\rho(\H)$, where $\H$ is the intersection of finitely many hyperplanes in $\R^n$. For brevity, we drop the adjectives `asymptotic' and `lattice' and simply refer to $\rho(\Delta)$ as the density of visible points in $\Delta$.

The left hand side of Figure \ref{fig:den-def} illustrates the definition of the density of visible points on a line in 2-space. Many of our arguments will use a linear transformation that puts the hyperplane(s) in a nicer position, while skewing the bounding boxes, $[-r,r]^n$. This is illustrated on the right hand side of Figure \ref{fig:den-def}.  
\begin{figure}
    \centering
    \includegraphics[width=\linewidth]{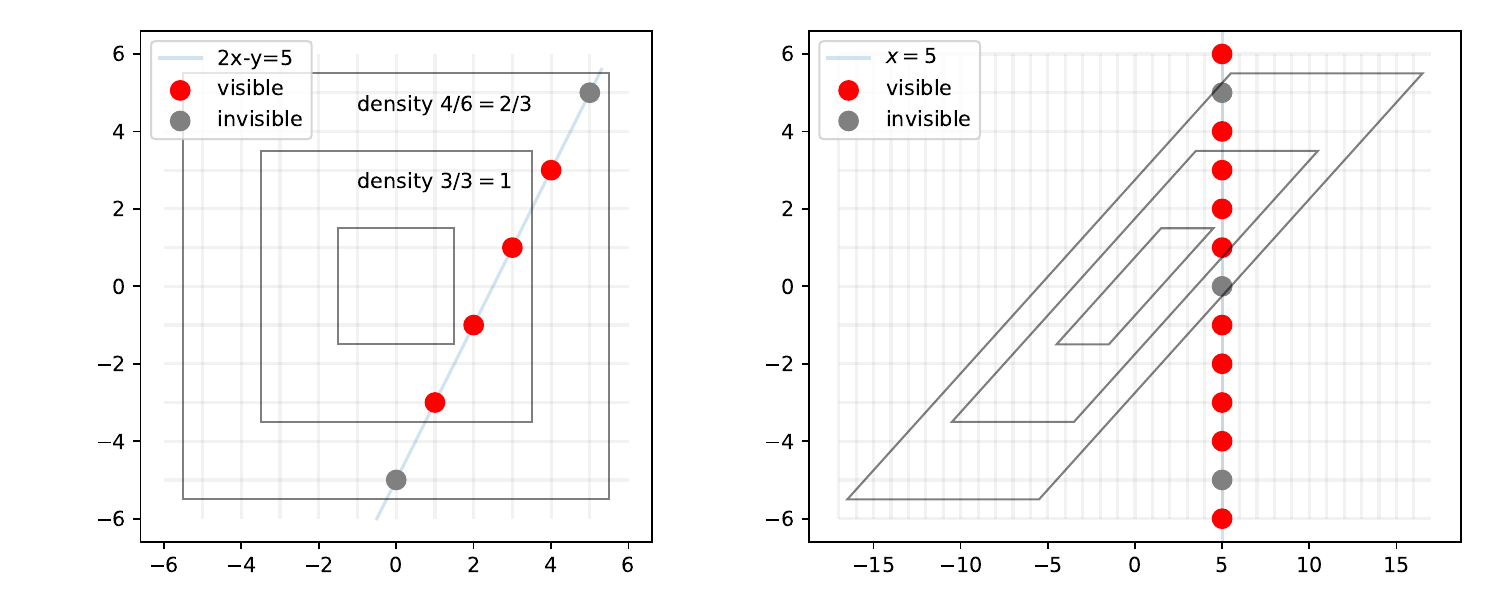}
    \caption{On the left is an illustration of bounding boxes used in the definition of asymptotic density of visible points in a hyperplane. On the right the picture has been transformed by the matrix $\begin{pmatrix}
        2&-1\\1&0
    \end{pmatrix}$ taking the line $2x-y=5$ to the line $x=5$ and shows the transformed bounding regions.}
    \label{fig:den-def}
\end{figure}
There is some flexibility in the bounding regions used to compute density; alternative regions have been used and studied in \cite{PrimPtsPolygons}, \cite[Thm 459]{HardyWright} and \cite{Counting-Problems}, for example. In particular, we use the following theorem from \cite{Counting-Problems} which we state here only in the special case where $f$ is the constant monotone function $f=1$:
\begin{theorem}[Theorem 2.1 of \cite{Counting-Problems} with $f=1$]
\label{counting density thm} Let $\Delta$ be a convex domain (compact and full dimension) of $\R^n$ of diameter $\omega \geq 1$ and let $\Delta_0$ be the subset of visible points in $\Delta$. Then 
\[\left| |\Delta_0|-\frac{\vol(\Delta)}{\zeta(n)}\right| = O\left(\begin{array}{ll}\omega\log\omega,&\textrm{if }n=2\\ \omega^{n-1},&\textrm{otherwise}\end{array}\right)\]
where the constant implied by the $O$-notation depends only on the dimension $n$.
\end{theorem}

The diameter of a set is the largest distance between points in the set 
and as noted in \cite[(2)]{Counting-Problems}, $\vol(\Delta) = O(\omega^n)$. The classic density theorem mentioned above (\cite{Nymann}) is obtained from Theorem \ref{counting density thm} by taking $\Delta = [-r,r]^n$ with $\vol(\Delta) = (2r)^n$, diameter $\omega = 2r\sqrt n $ and $|\Delta_1| =(2r+1)^n$. Dividing both sides of the equation from Theorem \ref{counting density thm} by $(2r+1)^n$ and letting $r \to \infty$ shows the density of visible points is $\rho(\R^n)=1/\zeta(n)$. 

In Section \ref{sec2}, Theorem \ref{our bigO}, we extend Theorem \ref{counting density thm} to give the density of points coprime to a fixed integer $b$ rather than the density of all visible points. The constant replacing $\zeta(n)$ in the equation is related to the Jordan totient of $|b|$, a variant of Euler's totient function. This generalization is used to prove the density of visible points on all hyperplanes and their intersections in Theorem \ref{snfdensity}.

For $n,b\in\Z_{>0}$, the Jordan totient function, $J_n(b)$, is defined to be the number of integer $n$-tuples selected from a complete set of residues modulo $b$ with $\gcd$ coprime to $b$ (\cite[pg 155 - 174]{Dickson}, \cite{TotientQuotients}). The Jordan totient function has closed formula \[J_n(b) = \sum_{d|b}\mu(d)\frac{b^n}{d^n} = b^n\prod_{p|b}1-\frac1{p^n}\] where the product is over all distinct primes dividing $b$ (see e.g., \cite[Ch 2, exercise 17]{Apostol}). 

Though we provide a rigorous argument in Section \ref{sec3}, we first present here a heuristic argument using the Jordan totient as to why the density of visible points on $\H$, the hyperplane defined by $x_1 = b$, should be $J_{n-1}(b)/b^{n-1}$. By its above definition, $J_{n-1}(b)$ is the number of visible lattice points in $\{b\}\times [0,b-1]^{n-1}\subset\H$. Translations of this set by linear combinations of $b\vec e_i$ for $2\leq i\leq n$, where $\vec e_i$ is the $i$-th unit normal basis vector, form a regular tiling of $\H$. Each tile has exactly $J_{n-1}(b)$ visible lattice points and $b^{n-1}$ total lattice points, meaning they each have a ratio of visible to total lattice points equaling $J_{n-1}(b)/b^{n-1}$. Therefore it is expected that the density of visible points in $\mathcal H$ is
\[\rho(\H)=\frac{J_{n-1}(b)}{b^{n-1}} = \prod_{p|b}1-\frac1{p^{n-1}}.\]
We prove this result for hyperplanes of the form $x_1=b$ in Lemma \ref{easy hyperplane} and for all hyperplanes $\H\subset\R^n$ defined by $\vec a\cdot \vec x=\sum_{i=1}^na_ix_i=b$ in Corollary \ref{generalhyperplanes} and a corresponding result for intersections of hyperplanes in Theorem \ref{snfdensity}.

Note here that we only concern ourselves with hyperplanes that are defined by $\vec a\cdot \vec x=b$ with $\vec a \in \Z^n$ and $\gcd(\vec a)=1$. If $\gcd(\vec a)\neq 1$ and $\gcd(\vec a)\nmid b$, then $\H$ contains no integral points and $\rho(\mathcal H)$ does not exist. If $\gcd(\vec a)\mid b$, then divide the equation by $\gcd(\vec a)$ and proceed.

Finally, we state here another result of \cite{Counting-Problems} relating the number of integral points in $\Delta$ to the domain's volume that we will use in Sections \ref{sec2} and \ref{sec3}:
\begin{lemma}
    \label{counting lemma}
    (\cite[Lemma 2.3]{Counting-Problems})
    Under the same hypothesis of Theorem \ref{counting density thm} and $\Delta_1$ the set of integral points of $\Delta$,
    \[\left||\Delta_1|-\vol(\Delta)\right| = O\left(\omega^{n-1}\right).\] 
\end{lemma}

\section{A generalization and some lower bounds}
\label{sec2}
Theorem \ref{our bigO} generalizes Theorem \ref{counting density thm}, allowing us to measure the density of points in a given domain that are coprime to a fixed integer $b$ rather than visible (coprime to every integer). We make use of the well-known identity
\begin{equation}
    \sum_{d|b}\frac{\mu(d)}{d^n} = \prod_{p|b} 1-\frac{1}{p^n}
    \label{e5}
\end{equation}
(\cite{Apostol})
as a finite version of the Euler product for $1/\zeta(n)$ used in the proof of Theorem $\ref{counting density thm}$.

Let $\Delta$ be a convex domain of full dimension in $\R^n$ with diameter $\omega$. As in Section \ref{sec1}, $\Delta_1$ is the set of integral points in $\Delta$ and $\Delta_b\subseteq \Delta_1$ the subset of points coprime to $b$.
\begin{theorem}
    \label{our bigO}
    For $b \ne 0$,
    \[\left||\Delta_b| - \left(\prod_{p|b}1-\frac{1}{p^{n}}\right)\vol(\Delta)\right| =O(\omega^{n-1})
    \]
    where the product is over all distinct primes dividing $b$.
\end{theorem}
    Note that if $b=0$ then $\Delta_b=\Delta_0$ is the subset of visible points in $\Delta$ and the product in the statement of Theorem \ref{our bigO} is infinite and equal to $1/\zeta(n)$. The theorem holds in this case by Theorem \ref{counting density thm}, except for $n=2$, where the bound is $\omega\log\omega$. The proof of Theorem \ref{our bigO} is somewhat simpler than that of Theorem \ref{counting density thm} because the product is finite but otherwise follows nearly identically.
    \begin{proof}
    The proof follows the argument of \cite[proof of 2.1]{Counting-Problems}. For any $d \in \Z_{>0}$, let $\Lambda_d=\Delta_1\cap(d\Z)^n$ be the subset of points with each component divisible by $d$. Then
    \begin{align}
    |\Delta_b| = \sum_{d|b} \mu(d)|\Lambda_d|.\label{e4}
    \end{align}
    To compute $|\Lambda_d|$ for $d|b$ let $\Delta/d = \{\x \in \R^n:d\x \in \Delta\}$ so that $|\Lambda_d| = |(\Delta/d)_1|$. $\Delta/d$ is a convex domain of full dimension and diameter $\omega/d$. Hence by Lemma \ref{counting lemma} we have
\begin{align}
 O(\omega^{n-1})=O((\omega/d)^{n-1})&=\Big||(\Delta/d)_1|-\vol(\Delta/d)\Big| = \Big||(\Delta/d)_1|-\vol(\Delta)/d^n\Big|.\label{e66}
\end{align}
Use \eqnref{e5}, \eqnref{e4}, \eqnref{e66} and the triangle inequality to sum on the divisors of $b$:
\begin{align}
    \left||\Delta_b|-\left(\prod_{p|b}1-\frac{1}{p^n}\right)\vol(\Delta)\right|
    &=\left|\sum_{d|b} \mu(d)|\Lambda_d|-\sum_{d|b}\frac{\mu(d)}{d^n}\vol(\Delta)\right|\nonumber\\
    &=\left|\sum_{d|b}\mu(d)\left(|\Lambda_d|-\vol(\Delta)/d^n\right)\right|\nonumber\\
    &\leq\sum_{d|b}|\mu(d)|\Big||(\Delta/d)_1|-\vol(\Delta)/d^n\Big|\nonumber\\
    &= \sum_{d|b}|\mu(d)|O(\omega^{n-1})= O(\omega^{n-1}).\label{e67}
\end{align}
\end{proof} 

In Section \ref{sec3}, the convex domains of Theorem \ref{our bigO} are hyperplane sections of $n$-dimensional cubes or transformations of them by $\gl_n(\Z)$. While the combinatorial type of the polyhedra which is the intersection of a $n$-cube and a hyperplane can be complicated, a general formula for the volume is known (\cite{Frank}) and there is also a formula for the number of lattice points in the intersection (\cite{Abel}). We need something similar but more basic in order to use Theorem \ref{our bigO} to attain the density of visible points on hyperplanes; we need a lower bound on the size of $\Delta_1(r)=\H\cap[-r,r]^{n}\cap\Z^n$, where $\H\subset\R^n$ is a fixed hyperplane. This lower bound is Corollary \ref{lower bound}. It follows from Lemma \ref{general lower bound} which gives a lower bound on the number of integral points on intersections of hyperplanes and is used again in the proof of Theorem \ref{snfdensity}.

\begin{lemma}
    \label{general lower bound}
    Let $\H_i$, $1\leq i \leq s$, be hyperplanes in $\R^n$ given by equations $\sum_{j=1}^na_{ij}x_j = b_i$ and set $A = (a_{ij}) \in M_{s,n}(\Z)$. Let $\Delta(r) = \bigcap_{i=1}^s\H_i \cap[-r,r]^n$ and $\Delta_1(r)\subset \Delta(r)$ the subset of integral points. Let $r_0$ be the smallest positive integer such that $\Delta_1(r)\ne \emptyset$. For $r\geq r_0$,
    \begin{align}
        |\Delta_1(r)| \geq \left(2\left\lfloor \frac{r-r_0}{c}\right\rfloor+1\right)^{n-m}
    \end{align}
    where $m=\mathrm{rank}(A)$, the null-space of $A$ is spanned by $\vec v_j = (v_{j1},\ldots,v_{jn})$ for $0\leq j \leq n-m$ and $c=\max(\sum_{j=1}^{n-m}|v_{jk}|)_{k=1}^{n}$. In particular, $|\Delta_1(r)| = \Omega(r^{n-m})$.
\end{lemma}
\begin{proof}
    Let $\p \in \Delta_1(r_0)$. All other integral points on $\bigcap_{i=1}^s \H_i$ are of the form $\p+\sum_{j=1}^{n-m}c_j\vec v_j$ for $c_j \in \Z$. When $|c_j|\leq \left\lfloor\frac{r-r_0}{c}\right\rfloor$ for all $j$, the $k$th component of $\p+\sum_{j=1}^{n-m}c_j\vec v_j$ has bounded magnitude:
    \begin{align*}
        |p_k+\sum_{j=1}^{n-m} c_jv_{jk}|&\leq |p_k|+\sum_{j=1}^{n-m} |c_j||v_{jk}|\\
        &\leq r_0+\left\lfloor\frac{r-r_0}{c}\right\rfloor\sum_{j=1}^{n-m}|v_{jk}|\\
        &\leq r_0+r-r_0=r
    \end{align*}
    That is, $\p+\sum_{j=1}^{n-m}c_j\vec v_j \in [-r,r]^n$, showing the lower bound in the lemma. For $r\gg r_0$, $\left\lfloor \frac{r-r_0}{c}\right\rfloor^{n-m} \geq \left(\frac{r-r_0}{c}-1\right)^{n-m} = \frac{1}{c^{n-m}}(r^{n-m}+\cdots)$, hence $|\Delta_1(r)| = \Omega(r^{n-m})$.
\end{proof}

In the case of only one hyperplane there is a nicer formulation for the constant $c$ from Lemma \ref{general lower bound}.
\begin{corollary}
    \label{lower bound}
    Let $\H$ be the hyperplane in $\R^{n}$ given by equation $\sum_{i=1}^n a_ix_i=b$. Let $\Delta(r) = \H\cap [-r,r]^{n}$ and $\Delta_1(r)\subset \Delta(r)$ the subset of integral points. Let $r_0$ be the smallest positive integer such that $\Delta_1(r_0)\neq\emptyset$. For $r\geq r_0$, 
    \[
    |\Delta_1(r)| \geq \left \lfloor \frac{r-r_0}{\sum |a_i|}\right\rfloor^{n-1}.
    \]
\end{corollary}

\begin{proof}
    Without loss of generality, assume $a_1 \ne 0$. The row space of 
    \begin{align}
        \begin{pmatrix}
            a_2&-a_1&0&\cdots&0\\
            a_3&0&-a_1&\cdots&0\\
            &&\ddots&&\\
            a_n&0&0&\cdots&-a_1
        \end{pmatrix}
    \end{align}
    is a basis of the null-space from Lemma \ref{general lower bound}, so that $c=\sum_{i=1}^n|a_i|$ satisfies the property of the $c$ from Lemma \ref{general lower bound}. Since the rank of $A$ is 1, the corollary follows. 
\end{proof}

\section{Density of Visible Points on Intersections of Hyperplanes}
\label{sec3}
Although Theorem \ref{snfdensity} ultimately subsumes the result, we begin this section with Lemma \ref{easy hyperplane} which demonstrates the proof technique using the simplest hyperplane: $x_1 = b$ in $\R^n$, where $b \in \Z$.  Theorem \ref{snfdensity} gives the density for arbitrary intersections of hyperplanes and the proof follows from arguments similar to those in Lemma \ref{easy hyperplane}, Lemma \ref{gln preserves}, and the results of Section \ref{sec2}. Several corollaries with simple formulas for the density of visible points follow. All hyperplanes and intersections of hyperplanes are assumed to contain at least one integral point and all defining equations $\vec a\cdot \x = b$ assume $\gcd(\vec a)=1$.

\begin{lemma}\label{easy hyperplane}
    Let $b \in \Z$ and let $\H\subset\R^n$ be the hyperplane $x_1=b$, the density of visible points on $\H$ is
    \[\rho(\H)=\prod_{p|b}1-\frac1{p^{n-1}}.\]
\end{lemma}
\noindent
Note this value is equal to $\displaystyle \frac{J_{n-1}(|b|)}{|b|^{n-1}}$ when $b \ne 0$ and $1/\zeta(n-1)$ when $b=0$, as expected from the discussion in Section \ref{sec1}.
\begin{proof}
        For $b=0$, $\H=\{0\}\times\R^{n-1}$ and this is the classic result that the density of visible points in $\R^{n-1}$ is $1/\zeta(n-1)$, which matches the Euler product given in the lemma. For $b=\pm 1$, the density is 1 because every integer point on the hyperplane is visible and the lemma holds because the product is empty. It is left to show the lemma for $b>1$ since the density will be the same for $\pm b$. 

        Fix $b>1$, set $\Delta(r)=\H\cap[-r,r]^n$,  $\Delta_1(r)\subset \Delta(r)$ the subset of integral points and $\Delta_0(r)$ the subset of visible points. For $r\geq b$, the diameter of $\Delta(r)$ is $\omega_r = 2r\sqrt{n-1}$ so  $O(\omega_r^{n-2}) = O(r^{n-2})$.  Therefore, by Theorem \ref{our bigO},
    \[\left||\Delta_0(r)| - \left(\prod_{p|b}1-\frac{1}{p^{n-1}}\right)\vol(\Delta(r))\right|=\left||\Delta_0(r)| - \left(\prod_{p|b}1-\frac{1}{p^{n-1}}\right)(2r)^{n-1}\right|
    =O(r^{n-2}).
    \]
    Since $|\Delta_1(r)| = (2r+1)^{n-1}$, we see that 
    \[\rho(\H)=\lim_{r \to \infty}\frac{|\Delta_0(r)|}{|\Delta_1(r)|} = \left(\prod_{p|b}1-\frac{1}{p^{n-1}}\right)\lim_{r \to \infty} \frac{(2r)^n}{(2r+1)^n} = \prod_{p|b}1-\frac{1}{p^{n-1}}.
    \]
\end{proof}

To make the calculation of the density of visible points on arbitrary hyperplanes and their intersections easier, we act on them by $\gl_n(\Z)$. Lemma \ref{gln preserves} notes that the action of $\gl_n(\Z)$ preserves visibility of integral points. Though certainly well known, we include a proof here for completeness.

\begin{lemma}
\label{gln preserves}
Let $\x$ be a point in $\Z^n$. Then $\gcd({\x}) = \gcd(A\x)$ for any $A \in \gl_n(\Z)$.  
\end{lemma}
\begin{proof}
    $\sl_n(\Z)$ is generated by transvection matrices $T_{ij}$ for $1\leq i\ne j\leq n$ with 1's on the diagonal and in the $(i,j)$ entry, and 0's elsewhere (\cite{MR1079696}). Moreover, $\gcd(T_{ij}\x)=\gcd(x_1,\ldots,x_{i-1},x_i+x_j,x_{i+1},\ldots,x_n)= \gcd(\x)$. Therefore, for any $A \in \sl_n(\Z)$, $\gcd(A\x) = \gcd(\x)$. If $A \in \gl_n(\Z)$ with $\det(A) = -1$, then $A = D(DA)$ with $D$ the diagonal matrix with entries $[-1,1,\ldots,1]$ and $DA \in \sl_n(\Z)$. Since $D$ has no effect on the $\gcd$, $\gcd(A\x) = \gcd(D(DA\vec X)) = \gcd(DA\x) = \gcd(\x)$.
\end{proof}

Let $\H_1,\ldots,\H_s$ be $s$ hyperplanes in $\mathbb R^n$ with a common integral point $\p = (p_1,\ldots,p_n)\in\Z^n$ and equations $\sum_{j=1}^na_{ij}x_i = \sum_{j=1}^na_{ij}p_i$ for $1\leq i \leq s$. Set $A$ to be the $s\times n$ matrix $A = (a_{ij})$ so that $\x \in \bigcap_{i=1}^k \H_i$ iff $A {\x} = A{\p}$.
 
 By the existence theorem for the Smith Normal Form (e.g., \cite[Thm. 1.11]{norman}) there exist matrices $U \in \gl_s(\Z)$ and $V \in \gl_n(\Z)$ so that $UAV = D$ where $D$ is a $s\times n$ matrix whose off diagonal entries are 0 and whose diagonal entries, $d_{i}$, are the invariant factors of the matrix $A$. That is, $d_{i}\mid d_{i+1}$ and $\prod_{j\leq i} d_{j}=\gcd(i\times i \textrm{ minors of }A)$ (\cite[Def 1.6]{norman}). Since $\gcd(a_{ij})=1$, we have $d_1 = 1$ and also $d_i=0$ for all $i>m$ where $m$ is the rank of $A$.

Since $\x \in \bigcap_{i=1}^s \H_i$ if and only if $A\x = A\p$, using the SNF matrices above, we see that $\x \in \bigcap \H_i$ if and only if 
$DV^{-1}\x = DV^{-1}\p$. Consider the change of variables $\y = V^{-1}\x$ and note that by Lemma \ref{gln preserves}, $\gcd(\x) = \gcd(\y)$ for all integral values of the $x_i$. Set $V^{-1}\p = (p_1',\ldots,p_n')$ so that $V^{-1}:\R^n \to \R^n$ transforms $\bigcap\H_i$ as follows:
\[
V^{-1}\left(\bigcap\H_i\right) =  \{\y\,|\,D\y = DV^{-1}\p\} = \{(p_1',\ldots,p_m',y_{m+1},\ldots, y_n)\,|\, y_i \in \Z\}.
\] 
Since $\gcd((p_1',\ldots,p_m',y_{m+1},\ldots, y_n)) = 1$ if and only if $\gcd(\gcd((p_1',\ldots,p_m'),y_{m+1},\ldots, y_n)=1$, from this formulation, Lemma \ref{easy hyperplane} and Lemma \ref{gln preserves}, it is clear that the density of visible points on the intersection of the hyperplanes should be the formula given in Theorem \ref{snfdensity}. The complication is that when we used the SNF to move the intersection of planes into a standard position, we also changed the standard bounding boxes into parallelotopes. The lower bounds given in Section \ref{sec2} allow us to prove the asymptotic density is not affected by this change.

\begin{theorem}
\label{snfdensity}
Let $A\x = A\p$ be a system of $s$ linear equations in $n$ variables, with Smith Normal Form matrices $D,U,V$ as above. The asymptotic density of visible points on the intersection of the $s$ hyperplanes is 
\[
\prod_{p|\gcd\left(p_1',\ldots,p_m'\right)}\left(1-\frac1{p^{n-m}}\right)\]
where $V^{-1}\p = (p_1',\ldots, p_n')$ and $m = \mathrm{rank}(A)$.
\end{theorem}
\begin{proof}
    Denote the $s$ hyperplanes by $\H_i$. Let $V^{-1}(r) = V^{-1}([-r,r]^{n})$ be the transformation of the $n$-cube of edge length $2r$ under $V^{-1}$. Let $\Delta(r) = V^{-1}(\bigcap \H_i \bigcap [-r,r]^n) = V^{-1}(\bigcap\H_i)\bigcap V^{-1}(r)$. Using the argument preceding the statement of the theorem, $\Delta(r)$ is the intersection of the $m$ hyperplane sections, $y_1=p_1',\,\ldots,\,y_m=p_m'$, with the $n$-parallelotope $V^{-1}(r)$. Let $\omega_r$ be the diameter of $\Delta(r)$, $\Delta_0(r)$ the visible points and $\Delta_1(r)$ the integral points of $\Delta(r)$. Since the first $m$ entries of any element in $\Delta(r)$ are $p_1',\ldots,p_m'$, we can identify $\Delta(r)$ with its projection onto the last $n-m$ entries. Under this projection $\Delta(r)$ is a convex domain in $\R^{n-m}$ and setting $b = \gcd(p_1',\ldots,p_m')$,  $\Delta_0(r)$ and $\Delta_1(r)$ become $\Delta_b(r)$ and $\Delta_1(r)$, respectively, from Theorem \ref{our bigO}. In particular, for $b \ne 0$, 
    \[\left|\frac{|\Delta_0(r)|}{|\Delta_1(r)|} - \left(\prod_{p|b}1-\frac{1}{p^{n-m}}\right)\frac{\vol(\Delta(r))}{|\Delta_1(r)|}\right| = \frac{O(\omega_r^{n-m-1})}{|\Delta_1(r)|}.
    \]

    $|\Delta_1(r)|$ equals the number of integral points in $[-r,r]^{n}\bigcap_{i=1}^s\H_i$. Therefore, by Lemma \ref{general lower bound}, $|\Delta_1(r)| = \Omega(r^{n-m})$. Let $\omega_1$ be the diameter of $V^{-1}(1)$ so that $r\omega_1$ is the diameter of $V^{-1}(r)$. Then $\omega_r \leq r\omega_1$, showing that $O(\omega_r^{n-m-1}) = O(r^{n-m-1})$. From this it follows that 
    \begin{align}
        \lim_{r \to \infty}\frac{O(\omega_r^{n-m-1})}{|\Delta_1(r)|} = 0
    \end{align}
    and hence the asymptotic density of visible points on $\bigcap \H_i$ is 
    \begin{align}
        \rho(\H)=\lim_{r \to \infty}\frac{|\Delta_0(r)|}{|\Delta_1(r)|} =\prod_{p|b} \left(1-\frac{1}{p^{n-m}}\right)\lim_{r \to \infty}\frac{\vol(\Delta(r))}{|\Delta_1(r)|}=\prod_{p|b} \left(1-\frac{1}{p^{n-m}}\right).
    \end{align}
    The last equality follows because $|\Delta_1(r)|=\Omega(r^{n-m})$ and  $\big||\Delta_1(r)|-\vol(\Delta(r))\big| = O(\omega_r^{n-m-1})$ by Lemma \ref{counting lemma}, therefore $\lim_{r \to \infty}\frac{\vol(\Delta(r))}{|\Delta_1(r)|} = 1$. 
    
    If $b=0$ then the theorem follows using the same argument with \cite[Thm 2.1 and Lemma 2.3]{Counting-Problems} in place of Theorem \ref{our bigO} and Lemma \ref{counting lemma}.
\end{proof}
We can now generalize Lemma \ref{easy hyperplane} to any hyperplane.
\begin{corollary}
    \label{generalhyperplanes}
    Let $\mathcal H$ be a hyperplane in $\R^{n}$ given by $\sum_{i=1}^{n} a_ix_i=b$, where $a_i,\,b \in \Z$ and $\gcd(a_1,\ldots,a_{n}) = 1$. The density of visible points on $\H$ is 
    \[
     \rho(\H)=\prod_{p|b}1-\frac{1}{p^{n-1}} = \frac{J_{n-1}(|b|)}{|b|^{n-1}}.\]
\end{corollary}
\noindent
When all of the invariant factors of $D$ are 1 in Theorem \ref{snfdensity} there is a more palatable formula for the density of the intersection of hyperplanes:
\begin{corollary}
\label{dm=1}
    Let $A\x = A\p$ define a system of $s$ linear equations in $n$ variables as above with Smith Normal Form matrices $D,U,V$ and the additional assumption that $d_m = 1$ where $m=\mathrm{rank}(A)$. Then the asymptotic density of visible points on the intersection of the $s$ hyperplanes is
    \[\prod_{p|\gcd\left(b_1,\ldots,b_s\right)}\left(1-\frac1{p^{n-m}}\right)\]
    where $A\p = (b_1,\ldots,b_s)$.
\end{corollary}
\begin{proof}
    By Theorem \ref{snfdensity} it is enough to show that $\gcd(p_1',\ldots,p_m') = \gcd(b_1,\ldots,b_s)$. Since $\gcd(b_1,\ldots,b_s) = \gcd (A\p)$, the result follows because 
    \[\gcd(A\p) = \gcd(UA\p) = \gcd(DV^{-1}\p) = \gcd(p_1',\ldots,p_m',0,\ldots, 0) = \gcd(p_1',\ldots,p_m').\]
\end{proof}

We end this section with a reformulation of the density of visible points on a hyperplane using $n$ integral points in general position in $\R^n$. 

Let $\{\p_i = (p_{i1},\ldots,p_{in})\}_{i=1}^n$, be $n$ points in $\R^n$ in general position. The unique hyperplane containing them has equation 
\[\det
\begin{pmatrix}
    p_{11}&\cdots&p_{nn}&1\\
    \vdots&\cdots&\vdots&1\\
    p_{n1}&\cdots&p_{nn}&1\\
    x_1&\cdots&x_n&1
\end{pmatrix}=0.
\]
The points being in general position means that not all $x_i$ has zero coefficient in this expression. Moreover, the constant term has absolute value the determinant of the matrix with rows $\p_1,\ldots,\p_n$, giving the following formulation:
\begin{corollary}
    \label{det_density}
    Let $\{\p_i = (p_{i1},\ldots,p_{in})\}_{i=1}^n$ be $n$ integral points in general position in $\R^n$. The asymptotic density of visible points on the hyperplane containing the $\p_i$ is 
    \[\prod_{p|d}1-\frac1{p^{n-1}}\]
    where $d$ is the determinant of the matrix with rows $\p_i$.
\end{corollary}

\section{Density of $k$-free Lattice Points}
\label{k-free sec}
In this section we extend the previous results to compute the density of so-called $k$-free lattice points. This amounts to loosening the sieve methods to detect only the $k$th-power of primes rather than the primes themselves.

For some $k\in\Z_{\geq 0}$, an integer $b$ is said to be {\it $k$-power-free} if there does not exist a prime $p$ such that $p^k|b$. A lattice point $\x\in\Z^n$ is $k$-power-free if $\gcd(\x)$ is $k$-power-free. We say that $\x$ is relatively $k$-power-free to $b$ if $p^k\nmid\gcd(\x,b)$ for all primes $p$. The $k$-power-free integers and lattice points are commonly referred to with the shortened adjective {\it $k$-free}, a convention we follow. These have been widely studied (e.g. \cite{k-free-tsvetkov}, \cite{Baake_2025}, and \cite{Pleasants_2013}) and notably, the asymptotic density of $k$-free lattice points in $\Z^n$ is known to be $1/\zeta(kn)$ (\cite{Pleasants_2013}, Corollary 1). 

In this section we compute the asymptotic density of $k$-free lattice points in hyperplanes and their intersections. Our methods for computing these densities follow those presented earlier, only we now tune our sieves to recognize the $k$th power of primes rather than the primes themselves. When $k=1$, the $k$-free lattice points are the visible lattice points, so the results in this section serve as generalizations to the theorems presented earlier.

We first prove a generalized version of Theorem \ref{our bigO}. The proof is a generalization of the proof of Theorem \ref{our bigO}, once again implementing a standard sieve argument only we now filter over the $k$th powers of primes that divide $b$ rather than the prime divisors. We introduce an ``order $k$'' M\"obius function that measures whether an integer $d$ is a product of $k$th powers of primes. Define 
\[\mu_k(d) =\begin{cases}
       \mu(m) &\textrm{if $d=m^k$}\\
       0 &\textrm{otherwise}
       \end{cases}\]
Note that $\mu_k$ is similar but not equal to Apostol's M\"obius function of order $k$ found in \cite{Apostol:order_k_mob}. Identity \eqref{e5} generalizes to $\mu_k$ as follows. Let $b$ be a non-zero integer and let $b_k$ be the product of all distinct primes $p$ with $p^k\mid b$. Note that for a divisor $d$ of $b$, $\mu_k(d) \ne 0$ if and only if $d=m^k$ and $m|b_k$.  In particular, for any $n\in\N$,
\begin{equation}\label{k-th mobius identity}
\sum_{d|b}\frac{\mu_k(d)}{d^n} = \sum_{m|b_k}\frac{\mu_k(m^k)}{m^{kn}} = \sum_{m|b_k}\frac{\mu(m)}{m^{kn}} = \prod_{p|b_k}1-\frac{1}{p^{kn}} = \prod_{p^k|b}1-\frac{1}{p^{kn}}
\end{equation}
where the second to last equal sign is \eqnref{e5}.

As in Section \ref{sec2}, let $\Delta$ be a convex domain of full dimension in $\R^n$ with diameter $\omega$ and integral points $\Delta_1$. Let $\Delta_b^k\subset\Delta_1$ be the subset of points relatively $k$-free to $b$. For any $d\in\Z^+$, recall that $\Lambda_d$ is the set of points in $\Delta_1$ with each component divisible by $d$. It follows that 
\begin{equation}\label{count mu sqrt k}
        |\Delta_b^k|=\sum_{d|b}\mu_{k}(d)|\Lambda_d|.
    \end{equation}
Using identities \eqref{k-th mobius identity} and \eqref{count mu sqrt k}, the proof of Theorem \ref{k-free} follows directly from that of Theorem \ref{our bigO}.
\begin{theorem}[Generalization of Theorem \ref{our bigO}]\label{k-free} For $b\neq 0$
        \[\left||\Delta_b^k| - \left(\prod_{p^k|b}1-\frac{1}{p^{kn}}\right)\vol(\Delta)\right| =O(\omega^{n-1})
    \]
\end{theorem}

Note that if $b=0$ then $\Delta_b^k=\Delta_0^k$ is the subset of $k$-free points in $\Delta$ and the product in the statement of Theorem $\ref{k-free}$ is infinite and equal to $1/\zeta(kn)$.

\begin{proof}
Substituting \eqref{k-th mobius identity} and \eqref{count mu sqrt k} in place of \eqnref{e5} and \eqnref{e4} in \eqnref{e67} completes the proof. 
\end{proof} 

We can now show that Theorem \ref{snfdensity} generalizes to a formula for the density of $k$-free lattice points on general hyperplanes and their intersections as expected. 

\begin{theorem}[Generalization of Theorem \ref{snfdensity}]\label{k-free density}
Let $A\x = A\p$ define a system of $s$ linear equations in $n$ variables with Smith Normal Form matrices $D,U,V$. The asymptotic density of $k$-free lattice points on the intersection of the $s$ hyperplanes is 
\[\prod_{p^k|\gcd\left(p_1',\ldots,p_m'\right)}\left(1-\frac1{p^{k(n-m)}}\right)\]
where $V^{-1}\vec p=(p'_1,...,p_n')$ and $m=\textrm{rank}(A)$.
\end{theorem}

\begin{proof}
    Let $V^{-1}(\bigcap\H_i) = \{(p_1',\ldots,p_m',y_{m+1},\ldots,y_n)|y_i \in \Z\}$, $V^{-1}(r) = V^{-1}([-r,r]^{n})$ and $\Delta(r) = V^{-1}(\bigcap\H_i)\cap V^{-1}(r)$ as in Theorem \ref{snfdensity}. Let $\Delta_0^k(r)\subset \Delta_1(r)$ denote the $k$-free points of $\Delta(r)$. Set $b = \gcd(p_1',\ldots,p_m')$. As in the proof of \ref{snfdensity}, projecting $\Delta(r)$ onto the last $n-m$ entries identifies $\Delta(r)$ as a convex domain in $\R^{n-m}$ and $\Delta_0^k(r)$ and $\Delta_1(r)$ with the $\Delta_b^k(r)$ and $\Delta_1(r)$ found in Theorem \ref{k-free}, respectively. In particular, for $b \ne 0$, 
    \[\left|\frac{|\Delta_0^k(r)|}{|\Delta_1(r)|} - \left(\prod_{p^k|b}1-\frac{1}{p^{k(n-m)}}\right)\frac{\vol(\Delta(r))}{|\Delta_1(r)|}\right| = \frac{O(\omega_r^{n-m-1})}{|\Delta_1(r)|}.
    \]
    where $\omega_r$ is the diameter of $\Delta(r)$. The proof now proceeds exactly as in Theorem \ref{snfdensity} to show $\lim_{r \to \infty}O(\omega_r^{n-m-1})/|\Delta_1(r)| = 0$ and $\lim_{r \to \infty}\vol(\Delta(r))/|\Delta_1(r)| = 1$. Hence the asymptotic density of $k$-free points on $\bigcap\H_i$ is 
    \[\lim_{r\to\infty}\frac{|\Delta_0^k(r)|}{|\Delta_1(r)|} = \prod_{p^k|b}1-\frac{1}{p^{k(n-m)}}.\]
\end{proof}

Applying this theorem to the case $m=1$ gives
\begin{corollary}
    \label{k-free general hyperplane}
    Let $\mathcal H$ be a hyperplane in $\R^{n}$ given by $\sum_{i=1}^{n} a_ix_i=b$, where $a_i,\,b \in \Z$ and $\gcd(a_1,\ldots,a_{n}) = 1$. The density of $k$-free points on $\H$ is 
    \[\prod_{p^k|b}1-\frac{1}{p^{k(n-1)}}.\]
\end{corollary}


\section{Structure of Density Sets}\label{sec:ex}
We conclude with an investigation into the structure of the set of possible densities of visible points on hyperplanes for a fixed dimension. For each $n\geq 1$, define the density set to be
\[D_n=\bigg\{\prod_{p|b}1-\frac{1}{p^n}:b\in\Z\bigg\}\subset\left[\frac{1}{\zeta(n)},1\right].\]
$D_n$ is also the range of the Jordan totient quotient ${J_n(b)}/{b^n}$, which 
by Corollary \ref{generalhyperplanes} is the set of possible densities of visible lattice points on a hyperplane in $\R^{n+1}$. Note that a density of ${1}/{\zeta(n)}$ is achieved by hyperplanes of the form $\vec a\cdot\vec x=0$ and a density of $1$ is achieved by hyperplanes of the form $\vec a\cdot \x=\pm 1$. We are interested in $D_n$ and its closure, $\overline D_n$, in the interval $\big[{1}/{\zeta(n)},1\big]$.

For $n=1$, $\big[{1}/{\zeta(n)},1\big] = [0,1]$. The densities of lines in $\R^2$ are of the form $\prod_{p|b}1-p^{-1}=\frac{\varphi(b)}{b}$. Not every rational density will occur because when in lowest terms these rational numbers will all have square free denominators, so there is no line with density $1/4$, say. However, Schinzel and Sierpi\'nski (\cite{Schinzel}, see also \cite[pg. 38]{Ribenboim}) proved that the set $\left\{{\varphi(b)}/{b}\,|\, b \in \N\right\}$ is dense in $(0,1)$, so $\overline D_1 = [0,1]$ and therefore, for any $\alpha \in [0,1]$ we can find a line with asymptotic density as close as we like to $\alpha$. 

\begin{ex}
\label{ex1}
    Using the greedy algorithm, the first three primes used in the approximation of $\frac1\pi= .3183\ldots$ are $2, 3, 23$, giving the three approximations $(1-\frac12) = \frac12, \frac12(1-\frac13) = \frac13$, $\frac13(1-\frac1{23}) = \frac{22}{69} = .3188\ldots$. Set $b_1=2$, $b_2 = 2\cdot 3$ and $b_3 = 2\cdot 3\cdot 23$. By Corollary \ref{generalhyperplanes}  equations $(b_i-1)x+b_iy = b_i$, plotted in Figure \ref{fig:lines approx}, have visible point densities equal to the three approximations above, respectively. The coefficient $b_i-1$ of $x$ is used so that the $\gcd$ of the $x$ and $y$ coefficients is 1. Taking the limit as the number of primes used goes to infinity shows that these  lines, whose visible point densities approach $\frac1\pi$ as the number of primes increases, themselves approach $y=1-x$. This limit line is not unique to $\frac1\pi$ as any $\alpha \in (0,1)$ with infinite product $\alpha = \prod 1-1/p$ can be expressed in the same way.
\begin{figure}
    \centering
    \includegraphics[width=.4\linewidth]{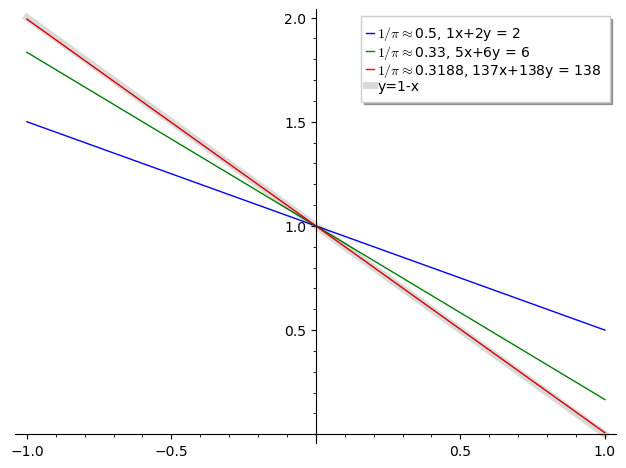}
    \caption{Three lines whose visible point densities successively approximate $\frac1\pi$ from Example \ref{ex1}.}
    \label{fig:lines approx}
\end{figure}
\end{ex}

In the remainder of this section we show that $D_n$ is not dense in $\big[{1}/{\zeta(n)},1\big]$ for $n>1$, so the result of Schinzel and Sierpi\'nski that $\overline D_1=[0,1]$ does not generalize to higher dimensions (Theorem \ref{achievement set}). In Theorem \ref{finite union of intervals} we show that $\overline D_n$ is a finite union of compact intervals for all $n\geq 1$ and then in Lemma \ref{lem:5bound} find $\overline D_2$ precisely. 

\begin{theorem}\label{achievement set}
    For $n>1$, $1-\frac{1}{2^n}<\prod_{p\ne2}1-\frac{1}{p^n}$ and
    \[D_n \cap (1-\frac{1}{2^n},\prod_{p\ne2}1-\frac{1}{p^n}) = \emptyset.\]
    In particular, $D_n$ is not dense in the interval $[1/\zeta(n),1]$
\end{theorem}
\begin{proof}
    Since elements of $D_n$ look like $ \prod_{p|b}1-1/p^n$, the fact that the intersection of the two sets is empty is easy to see; if $2|b$ then $ \prod_{p|b}1-1/p^n\leq 1-1/2^n$ and if $2 \nmid b$ then $ \prod_{p|b}1-1/p^n$ is a subproduct of $\prod_{p\ne 2}1-1/p^n$ and hence greater than it.

    To show the inequality, multiply both sides by $1-1/2^n$ to see that it is enough to show $(1-1/2^n)^2<1/\zeta(n)$, or equivalently $\zeta(n)<(1-1/2^n)^{-2}$. Expressed as infinite series the inequality becomes 
    \[\sum_{k=1}^\infty \frac{1}{k^n} = 1+\frac{1}{2^n}+\frac{1}{3^n}+\cdots <1+\frac{2}{2^n}+\frac{3}{2^{n2}}+\frac{4}{2^{n3}}+\cdots = \sum_{k=1}^\infty \frac{k}{2^{n(k-1)}}\]
    In particular, it is enough to show the following bound on the tail of $\zeta(n)$:
    \begin{equation}
        \sum_{k\geq 3}\frac1{k^n} < \frac1{2^n}+\sum_{k\geq 3}\frac{k}{2^{n(k-1)}}.\label{e89}
    \end{equation} This is true for all $n\geq 2$ by the upper bound given by the integral: 
    \begin{equation}\sum_{k\geq 3}\frac{1}{k^n}<\int_2^\infty \frac1{x^n}\,dx = \frac1{n-1}\frac{1}{2^{n-1}}\label{e88}
    \end{equation}
    For $n\geq 3$ the right hand side of \eqnref{e88} is at most $1/2^n$ and hence the bound \eqnref{e89} is satisfied. When $n=2$ the bound in \eqnref{e88} is 1/2 whereas the right hand side of \eqnref{e89} is 19/36.
\end{proof}

Theorem \ref{achievement set} shows that for $n>1$, $D_n$ has at least one gap, so $\overline D_n \ne \left[1/\zeta(n),1\right]$. In Theorem \ref{finite union of intervals} we show that for every $n\geq 2$, $\overline D_n$ is a finite union of closed intervals. In particular, $\overline D_n$ is not homeomorphic to a cantor set or a cantorval, which are the other possibilities as described in \cite{Achievement_sets}. 

Logarithms translate the problem of understanding $D_n$ into understanding subsums of the series $\sum_p a_p$, where $a_p=-\log(1-p^{-n})$. Foundational results on subsums of series, also called achievement sets in the literature, can be found in \cite{Achievement_sets_3}, \cite{Achievement_sets_2}, \cite{kakeya1914partial} and \cite{Achievement_sets}. For a convergent series of positive real numbers, $\sum_{n\geq 1}a_n$, the achievement set is defined as $A(a_n)=\{\sum_{n\in S}a_n: S\subset\N\}$. We can understand the structure of $A(a_p)$ with the following theorem.

\begin{theorem}[Theorem 1.1 of \cite{Achievement_sets}]\label{finite union lemma}
    If $\sum_{n=1}^\infty a_n$ is absolutely convergent with infinitely many nonzero terms and $(a_n)$ is non-increasing, then $A(a_n)$ is a finite union of compact intervals if and only if $a_k\leq \sum_{n> k}a_n$ for all but finitely many $k$. 
\end{theorem}
Since $-\log D_n:=\big\{-\log x:x \in D_n\big\}=\big\{\sum_{p|b}-\log(1-p^{-n}):b\in\Z\big\} = \big\{\sum_{p|b}a_p:b \in \Z\big\}$, $-\log D_n$ is the subset of finite subsums of $A(a_p)$ together with $\log(\zeta(n))$, and its closure is thus $A(a_p)$. In particular, $\overline{D_n}$ is a finite union of compact intervals if and only if $A(a_p)$ is. 
Applying Theorem \ref{finite union lemma} to $A(a_p)$ we have that showing $\overline{D_n}$ is a finite union of compact intervals reduces to showing that for each $n$ there is a prime $p_n$ so that $a_p\leq \sum_{q>p}a_q$ for all $p\geq p_n$. 

For $|x|<1$ the following inequalities hold by considering the Taylor expansions:
\[x\leq -\log(1-x)\leq\frac{x}{1-x}.\]
In particular, these hold for $x={1}/{p^n}$ for $n>1$ and any prime $p$. Thus it suffices to show that there exists a prime $p_n$ such that for all primes $p\geq p_n$ we have 
    \begin{align}\frac{1}{p^n-1}\leq \sum_{q>p}\frac{1}{q^n}.\label{to_show}
    \end{align}
Using Abel's identity (\cite[Thm. 4.2]{Apostol}) the tail of the series can be expressed as 
\[ \sum_{q>p}\frac{1}{q^n} =n\int_p^\infty \frac{\pi(x)}{x^{n+1}}\,dx - \frac{\pi(p)}{p^n}.\]
The inequality $\frac{x}{\log x} < \pi(x)$, from \cite[Cor.2]{Rosser-Schoenfeld} for $x \geq 17$,
 gives for $p\geq 17$,

\begin{align*}
\sum_{q>p}\frac{1}{q^n}
&\geq n\int_p^\infty\frac{1}{x^n\log x}dx-\frac{\pi(p)}{p^{n}}\\
&=nE_1(\log p^{n-1}) -\frac{\pi(p)}{p^{n}}\nonumber
\end{align*}
where $E_1(z)=\int_z^\infty \frac{e^{-x}}{x}dx$ is the exponential integral function. The inequality $E_1(z)\geq \frac{1}{e^{z}(z+1)}$ (\cite{hopf1934mathematical} see also \cite[pg.229, 5.1.19]{Handbook}) decreases the lower bound:
\begin{equation}
    \sum_{q>p}\frac{1}{q^n}\geq \frac{n}{p^{n-1}(\log p^{n-1} + 1)} -\frac{\pi(p)}{p^{n}}\label{Hopf2}
\end{equation}
In Theorem \ref{finite union of intervals} and Lemma \ref{lem:5bound} we start with \eqnref{Hopf2} and then use different upper bounds on $\pi(x)$ to show $a_p\leq_{q>p}a_q$.
\begin{theorem}\label{finite union of intervals}
    For $n\geq 2$, $\overline{D}_n$ is a finite union of compact intervals. 
\end{theorem}
\begin{proof}
    By \cite[Thm.2]{Rosser-Schoenfeld} for $x \geq e^{3/2} \approx 4.5$, $\pi(x)<x/(\log x-3/2)$. Putting this into \eqnref{Hopf2} we get for $p\geq 17$
    \begin{equation}
        \sum_{q>p}\frac{1}{q^n}\geq \frac{n}{p^{n-1}(\log p^{n-1} + 1)} -\frac{1}{p^{n-1}\log p} = \frac{\log p -\frac{3}{2}n-1}{p^{n-1}(\log p^{n-1}+1)(\log p-\frac{3}{2})}.
    \end{equation}
We can now compare $\sum_{q>p}a_q\geq \sum_{q>p}\frac{1}{q^n}$ with $a_p\leq\frac{1}{p^{n-1}}$:
\[\frac{1}{a_p}\sum_{q>p}a_q \geq \frac{\sum_{q>p}\frac{1}{q^n}}{1/{(p^n-1)}} \geq \frac{(p^n-1)(\log p-\frac{3}{2}n-1)}{p^{n-1}(\log p^{n-1}+1)(\log p-\frac{3}{2})}.\]
Since the right hand side goes to infinity as $p \to \infty$ there exists a prime $p_n$ so that for all $p\geq p_n$, $a_p\leq \sum_{q>p}a_q$. 
\end{proof}

We conclude by explicitly computing the gaps in the set $D_2$, that is the set of possible densities of planes in $\R^3$. $D_2$ is shown in Figure \ref{fig:n=2 density} along with the endpoints of the empty intervals.
    
\begin{figure}[h]
\centering
\includegraphics[width=.6\linewidth]{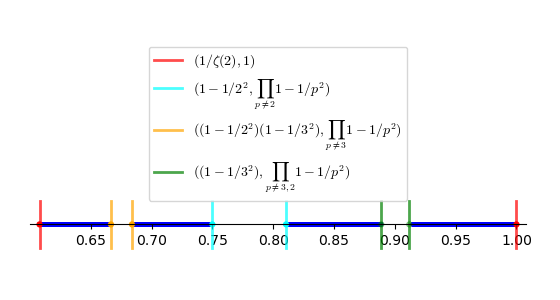}
\caption{Blue points are the values of $J_2(b)/b^2$ for all $b$ with primes in its prime factorization up to 61. The color matching vertical lines indicate empty intervals with labels as in the legend. The proof that these are the intervals of $\overline D_2$ is given in Theorem \ref{D2 intervals}.}
\label{fig:n=2 density}
\end{figure}

\begin{lemma}\label{lem:5bound}
Fix $n=2$ so that $a_p = -\log(1-1/p^2)$. For every prime $p \ge 5$, $
a_p \leq \sum_{q>p} a_q
$.
\end{lemma}

\begin{proof}
As in the proof of Theorem \ref{finite union of intervals}, it is enough to show $\sum_{q>p}1/q^2 \geq 1/(p^2-1)$ for all $p\geq 5$. By \cite[Cor.1]{Rosser-Schoenfeld}, $\pi(x)<{1.25506x}/{\log(x)}$. Putting this into \eqnref{Hopf2} we get, for $p\geq 17$, 
\begin{align} \sum_{p<q}\frac{1}{q^2} &\geq \frac{2}{p(\log p + 1)}-\frac{1.25506}{p\log p}\nonumber\\
&\geq \frac{2}{p(\log p + A\log p)}-\frac{1.25506}{p\log p}\label{e888}\\
&= \frac{.2}{p\log p}.\nonumber
\end{align}
Here $A = 2/1.45506-1$ is chosen so that $A\log p\geq 1$ for $p\geq 17$ so that inequality \eqnref{e888} holds. It is not hard to check that $.2/(p\log p)>1/(p^2-1)$ for all $p\geq 17$, so the desired inequality holds for all $p\geq 17$. That leaves the primes $p=5,7,11,13$. It is easiest to check $\exp(a_p)\leq \exp(\sum_{q>p}a_q)$ by hand for these primes. For example, in the $p=5$ case this is equivalent to checking $(1-1/4)(1-1/9)(1-1/25)^2=384/625\geq 6/\pi^2 = 1/\zeta(2)$.
\end{proof}
The inequality from Lemma \ref{lem:5bound} fails to hold for $p=2,3$ because $(1-1/4)^2<6/\pi^2$ and $(1-1/4)(1-1/9)^2<6/\pi^2$. As found in the ``First Gap Lemma'' (\cite[21.9]{subsums_Bartoszewicz}) these inequalities tell us exactly where the gaps are in $D_2$. 
\begin{theorem}
\label{D2 intervals}
\[\overline D_2 = \big[\frac{6}{\pi^2},\frac{2}{3}\big]\cup\big[\frac{27}{4\pi^2},\frac{3}{4}\big]\cup\big[\frac{8}{\pi^2},\frac{8}{9}\big]\cup\big[\frac{9}{\pi^2},1\big]\]
\end{theorem}
\begin{proof}
    By Lemma \ref{lem:5bound}, for $p\geq 5$, $a_p\leq\sum_{q>p}a_q$. Thus, by \cite{kakeya1914partial} every element in the interval $[0,\sum_{q\geq 5}a_q]$ is a subsum of the series $\sum_{q\geq 5}a_q$. Exponentiating, we see that the partial products of $\prod_{q\geq 5}(1-1/q^2)$ are dense in $[\prod_{p\ne2,3}1-1/p^2,1] = [9/\pi^2,1]$ and hence $\overline D_2\supset [9/\pi^2,1]$.

    By adding $a_2$ to the subsums of $\sum_{q\geq 5}a_q$, we see that every element in $[a_2,a_2+\sum_{q\geq 5}a_q]$ is a subsum of the series $\sum a_q$. In particular, 
    \[\overline D_2\supset  \big[\prod_{p\ne3}1-1/p^2,3/4\big] = \big[\frac{27}{4\pi^2},\frac{3}{4}\big]\]
    Adding $a_3$ or $a_2+a_3$ to the interval behaves accordingly, showing $\overline D_2$ contains the sets. On the other hand, any element in $\overline D_2$ is in one of these sets. Which one is determined by whether or not the product includes the primes 2 and 3.
\end{proof}

\bibliographystyle{plain} 
\bibliography{refs} 

\end{document}